\documentclass[11pt]{amsart}
\usepackage{amssymb,amsmath,amsthm,amsfonts}
\usepackage[mathscr]{euscript}
\usepackage{dsfont}
\usepackage{graphicx}
\usepackage{mathrsfs}
\usepackage{esint}
\usepackage{xcolor}
\usepackage{subfig}
\usepackage{subfloat}
\usepackage[usenames,dvipsnames]{pstricks}
\usepackage{epsfig}
\usepackage{pst-grad} 
\usepackage{pst-plot}
\usepackage{color}
\definecolor{citation}{rgb}{0.2,0.5,0.2}
\definecolor{formula}{rgb}{0.1,0.2,0.5}
\definecolor{url}{rgb}{0,0.2,0.7}
\usepackage[colorlinks=true,linkcolor=formula,urlcolor=url,citecolor=citation]{hyperref}
\allowdisplaybreaks[4]
\usepackage{color}
\definecolor{citation}{rgb}{0.2,0.5,0.2}
\definecolor{formula}{rgb}{0.1,0.2,0.5}
\definecolor{url}{rgb}{0,0.2,0.7}
\usepackage[colorlinks=true,linkcolor=formula,urlcolor=url,citecolor=citation]{hyperref}
\allowdisplaybreaks[4]

\def\ds{\displaystyle}

\def\l{{\mathcal{L}}}
\def\m{{\mathcal{M}}}
\newcommand{\R}{\mathbb{R}}

\newcommand{\N}{\mathbb{N}}
\newcommand{\D}{\mathrm{d}}

\newcommand{\eps}{\varepsilon}

\renewcommand{\epsilon}{\varepsilon}

\newcommand{\opl}[1]{\l[{#1}]}
\newcommand{\opm}[1]{\m[{#1}]}

\newcommand{\baa}{\begin{array}}
\newcommand{\eaa}{\end{array}}
\newcommand{\ba}{\begin{eqnarray}}
\newcommand{\ea}{\end{eqnarray}}
\newcommand{\be}{\begin{equation}}
\newcommand{\ee}{\end{equation}}
\newtheorem{theorem}{Theorem}[section]

\newtheorem{lemma}[theorem]{Lemma}

\newtheorem{claim}[theorem]{Claim}
\theoremstyle{definition}

\theoremstyle{remark}
\newtheorem{remark}[theorem]{Remark}
\theoremstyle{remark}

\renewcommand{\le}{\leqslant}
\renewcommand{\leq}{\leqslant}
\renewcommand{\ge}{\geqslant}
\renewcommand{\geq}{\geqslant}

\def\ackno#1{\par\medskip{\small\textbf{Acknowledgements.}~#1}}

\newlength{\defbaselineskip}
\title{A Note on Liouville type results for a fractional obstacle problem}
\author[J\'er\^{o}me Coville]{J\'er\^ome Coville}
\address{BioSP, INRA, 84914, Avignon, France}

\thanks{The author has been supported by the ANR DEFI project NONLOCAL (ANR-14-CE25-0013). The author want to thank professor Changfeng Gui for bringing to my attention this question during the Matrix event ``Recent Trends on Nonlinear PDEs of Elliptic and Parabolic Type''. These results have emerged through the scientific discussions during this MATRIX event.}
\begin{document}

\maketitle

\begin{abstract}
This note is a synthesis of my reflexions on some questions  that have emerged  during the MATRIX event ``Recent Trends on Nonlinear PDEs of Elliptic and Parabolic Type'' concerning the qualitative properties of solutions to some non local reaction-diffusion equations of the form
$$
\opl{u}(x)+f(u(x))=0, \quad \text{ for } \quad x\in \overline{\R^n\setminus K},
$$
where  $K\subset\mathbb{R}^N$ is a bounded smooth compact ``obstacle", $\l$ is non local operator  and $f$ is a bistable nonlinearity. When $K$ is convex and  the nonlocal operator $\l$ is  a  continuous  operator of convolution type then some Liouville-type results for solutions satisfying some asymptotic limiting conditions at infinity have been recently established by Brasseur, Coville, Hamel and Valdinoci \cite{Brasseur2016}. Here, we show that for a bounded smooth convex obstacle $K$, similar  Liouville type results hold true  when the operator $\l$ is  the regional s-fractional Laplacian. 
\end{abstract}

\tableofcontents


\section{Introduction}
A classical topic in applied analysis consists in the study of ``diffusive processes'' in complex media e.g. media containing  obstacles. Roughly speaking, this corresponds to study  dispersal processes that follows a random motion in an environment that possess an inaccessible region. At the macroscopic level, this problem can be translated into a reaction - diffusion equation that is defined outside a set~$K$,
which acts as an impenetrable obstacle.

One of the cornerstones in the study of these processes lies in suitable rigidity results of Liouville-type,
which allow the classification of stationary solutions, at least under some geometric assumption on the obstacle~$K$.

In this note, we investigate further  non local version of a diffusion equation and provide new Liouville-type result (whose
precise statements will be given in Section~\ref{MAIN RES}).

Concretely, we will suppose that the random motion  is modelled by  a L\'evy flight which at the macroscopic level, leads to consider an integral operator with a singular positive kernel.  For such type of processes we will show that the solutions of the stationary equation with  a prescribed behaviour at infinity are necessarily constant, at least when the obstacle is convex.

We now provide the detailed mathematical description of the problem that we take into account.


\subsection{A fractional obstacle problem}\label{SE:intro:nonlocal}

Throughout this note, $K$ denotes a smooth compact set of $\R^n$ with $n\ge2$,  $\left|\cdot\right|$ denotes the Euclidean norm in $\R^n$ and  $\l$ denotes the  regional fractional nonlocal  operator (\cite{Guan2005,Guan2006})  defined for $s\in (0,1)$ by
\begin{equation}\label{DEF:L}
\opl{u}(x):=\ds{C_{n,s}\lim_{\epsilon\to 0} \left(\int_{|x-y|>\epsilon,y\in \R^n\setminus K} \,\frac{\big( u(y)-u(x) \big)}{|x-y|^{n+2s}}\,\D y\right)}.
\end{equation}

We are interested in qualitative properties of smooth bounded solutions to the following non local semilinear equation
\begin{align}
&\opl{u}(x)+f(u(x))=0 \quad\mbox{ for all  } x\in \overline{\R^n\setminus K}, \label{EQt}
\end{align}
where  $f$ is a $C^1$ ``bistable" non-linearity and when  necessary  with  the  Neumann boundary condition below  
\begin{align} 
&\nabla u(x)\cdot\nu(x)=0  \qquad\mbox{ for all  } x\in \partial K,  \label{EQtBC}
\end{align}
where $\nu(x)$ denotes the outer normal derivative of the set $K$. The precise assumptions on  $K, u, f$ and $\l$ will be given later on.
Typically, this homogeneous Neumann boundary condition is required to define properly  the regional fractional Laplacian $\l$ on the boundary of the obstacle  when $s\in \left(\frac{1}{2},1\right)$.

This problem may be thought of as a fractional version of the following problem
\be\label{eqlocale}\left\{\begin{array}{rl}
\Delta u+f(u)=0 & \text{in }\R^n\setminus K,\vspace{3pt}\\
\nabla u\cdot\nu=0 & \text{on }\partial K.\vspace{3pt}
\end{array}\right.
\ee
For problem~\eqref{eqlocale} with the local diffusion operator $\Delta u$, it was shown in~\cite{Berestycki2009} that there
exist a time-global classical solution $u(t,x)$ to the parabolic problem
\be\label{EQ:BHM}\left\{\baa{rcll}
\ds\frac{\partial u}{\partial t} & \!\!=\!\! & \Delta u+f(u) & \text{in }\R\times\overline{\R^n\setminus K},\vspace{3pt}\\
\nabla u\cdot\nu & \!\!=\!\! & 0 & \text{on }\R\times\partial K\eaa\right.
\ee
satisfying $0<u(t,x)<1$ for all $(t,x)\in\R\times\overline{\R^n\setminus K}$,
and a classical solution $u_\infty(x)$ to the elliptic problem
\begin{align}
\left\{
\begin{array}{rl}
\Delta u_\infty+f(u_\infty)=0 & \text{in }\overline{\R^n\setminus K},\vspace{3pt}\\
\nabla u_\infty\cdot\nu=0 & \text{on }\partial K, \vspace{3pt}\\
0\le u_\infty\leq 1 & \text{in }\overline{\R^n\setminus K}, \vspace{3pt}\\
u_\infty(x)\to1 & \text{as }|x|\to+\infty.
\end{array} \label{EQ:BHM0}
\right.
\end{align}
The function $u_{\infty}$ is a stationary solution of~\eqref{EQ:BHM} and it is actually obtained as the large time limit of $u(t,x)$, in the sense that $u(t,x)\to u_\infty(x)$ as $t\to+\infty$ locally uniformly in $x\in\overline{\R^n\setminus K}$. Under some geometric conditions on $K$ (e.g. if $K$ is starshaped or directionally convex, see~\cite{Berestycki2009} for precise assumptions) it is shown in~\cite[Theorems~6.1 and~6.4]{Berestycki2009} that solutions to~\eqref{EQ:BHM0} are actually identically equal to~$1$ in the whole set $\overline{\R^n\setminus K}$. This Liouville property shows that the solutions $u(t,x)$ of~\eqref{EQ:BHM} constructed in~\cite{Berestycki2009} then satisfy
\begin{align}
u(t,x)\underset{t\to+\infty}{\longrightarrow} 1 \qquad{\mbox{ locally uniformly in }}x\in\overline{\R^n\setminus K}. \label{longtime}
\end{align}

From an ecological point of view, such a results can be interpreted as follows. Let us consider that $u(t,x)$ represents the density of a population that moves according to a Brownian motion in a environment consisting of the whole space~$\R^n$ with a compact obstacle $K$ and that the demography of this population can be described by the nonlinear function $f$.  Then, the equation ~\eqref{EQ:BHM} can be understood as the evolution of this population int the region $\R^n\setminus K$. In this context,~\eqref{longtime} means that, at large time, the population tends to occupy the whole space.

Assuming now that the random movement  of the individuals follows, say, a reflected symmetric $\alpha$-stable  L\'evy process, then the resulting reaction-diffusion equation will be 
\begin{align}
&\frac{\partial u}{\partial t}(t,x)=\opl{u}(t,x)+f(u(t,x)) \quad \text{ for all }\quad t>0, x\in \R^n\setminus K   \label{NONLOCAL:EVOL:EQ}\\
& + \text{A Neumann type boundary condition when necessary}\nonumber
\end{align}
where$\l$ is the regional fractional Laplacian defined above. 

The  numerical simulations below (see Figure \ref{fig-obs}) obtained for a non singular version of the  fractional operator $\l$ namely when we replace the singular measure $\frac{1}{|z|^{n+2s}}$ by $\frac{1}{\delta+|z|^{n+2s}}$ for $\delta<<1$,  suggest that the long time  behaviour of solution of positive equation \eqref{NONLOCAL:EVOL:EQ} should be identical as those observed in the classical reaction-diffusion equation. 

\begin{figure}[h!]
 \centering
\subfloat[$t=0$]{\includegraphics[width=4.2cm, height=2.5cm]{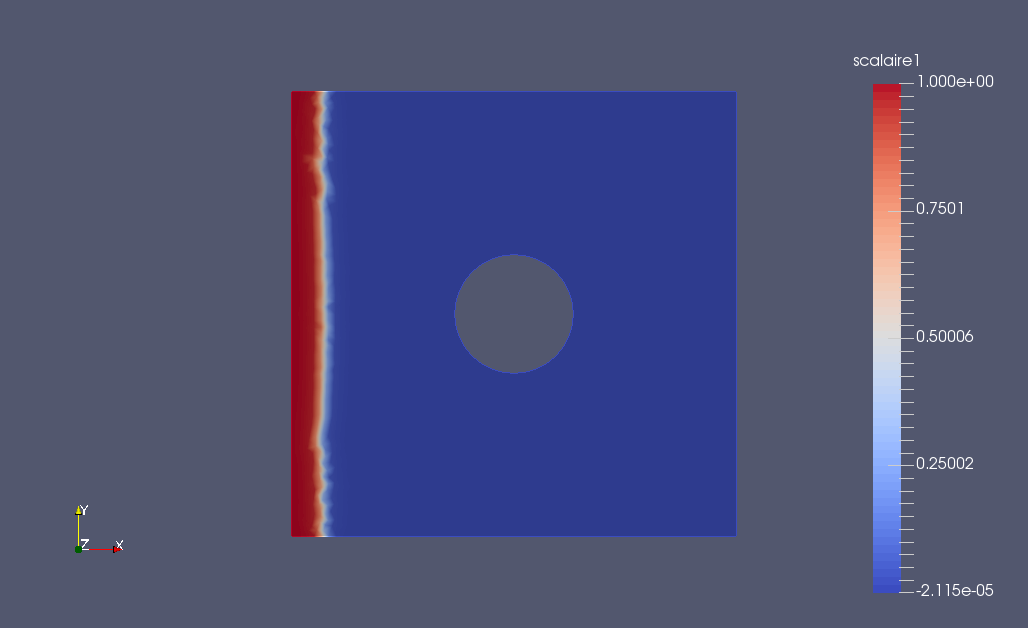}}
\subfloat[$t=40$]{\includegraphics[width=4.2cm, height=2.5cm]{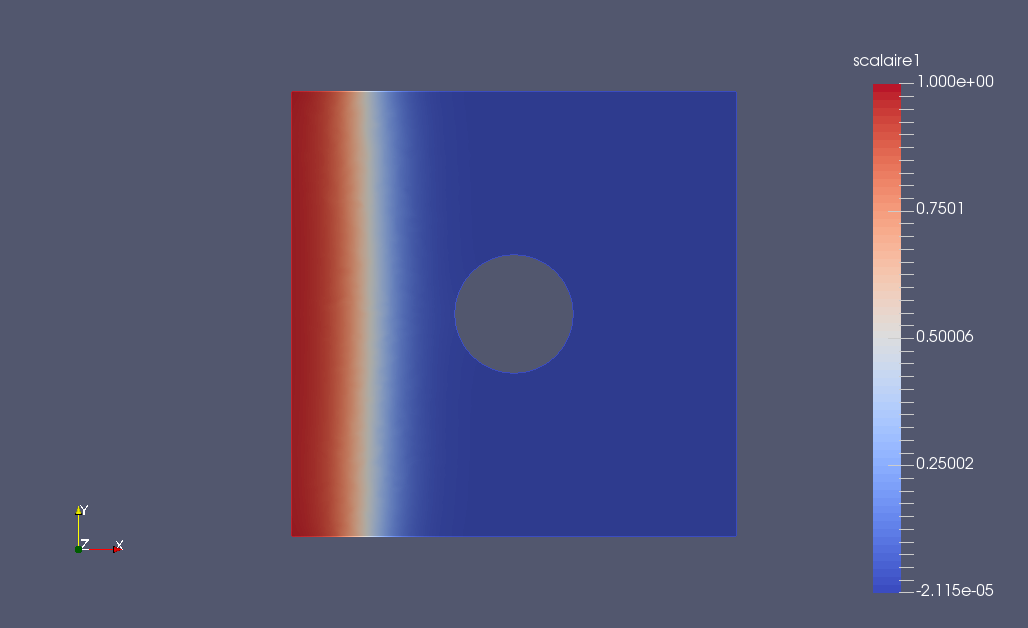}}
 \subfloat[$t=80$]{\includegraphics[width=4.2cm, height=2.5cm]{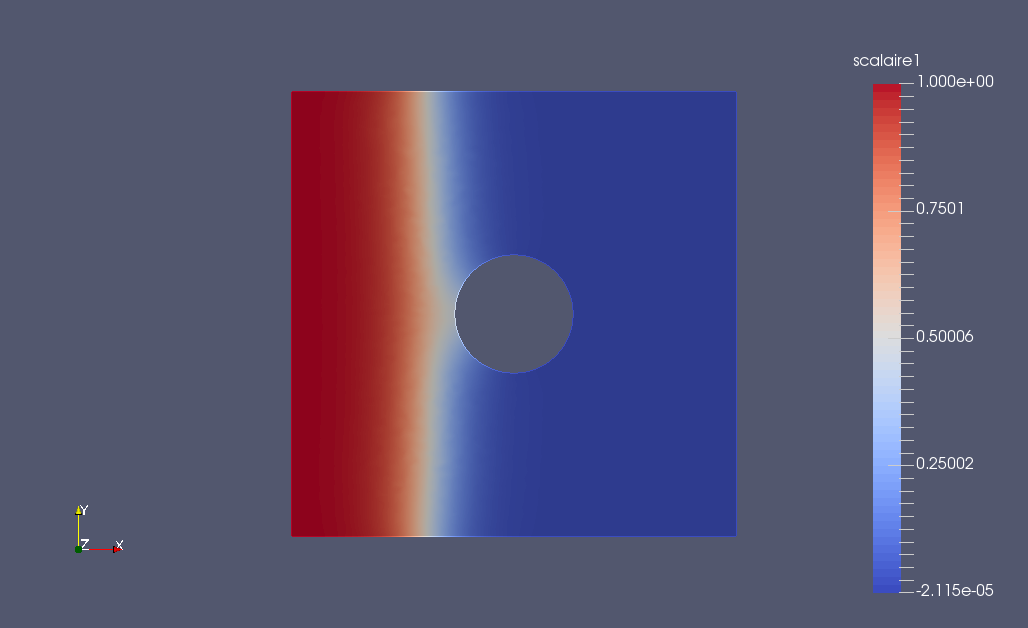}}
 \subfloat[$t=120$]{\includegraphics[width=4.2cm, height=2.5cm]{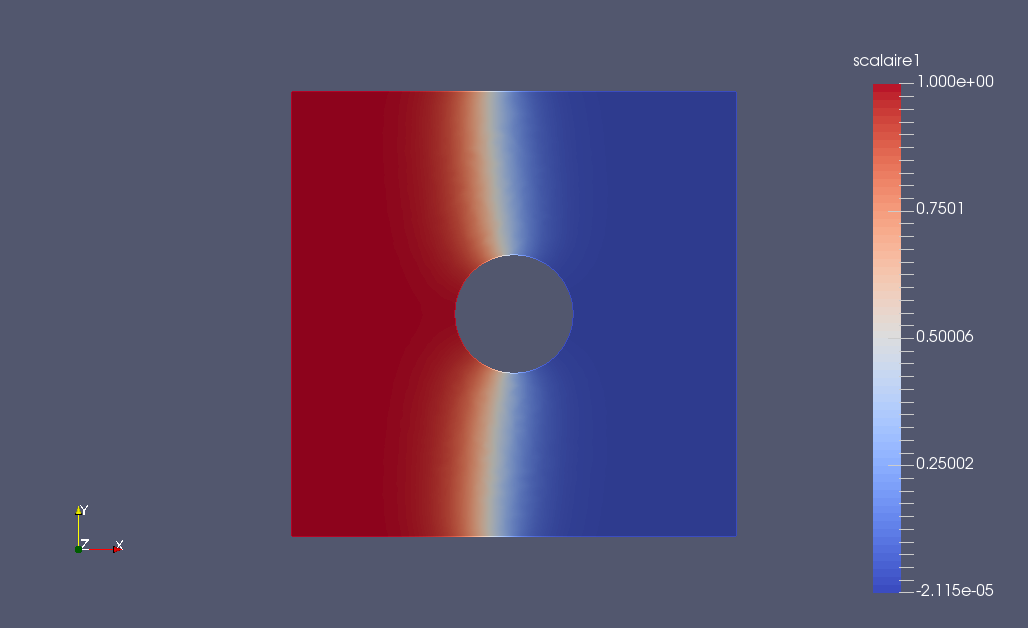}}
 
 \subfloat[$t=160$]{\includegraphics[width=4.2cm, height=2.5cm]{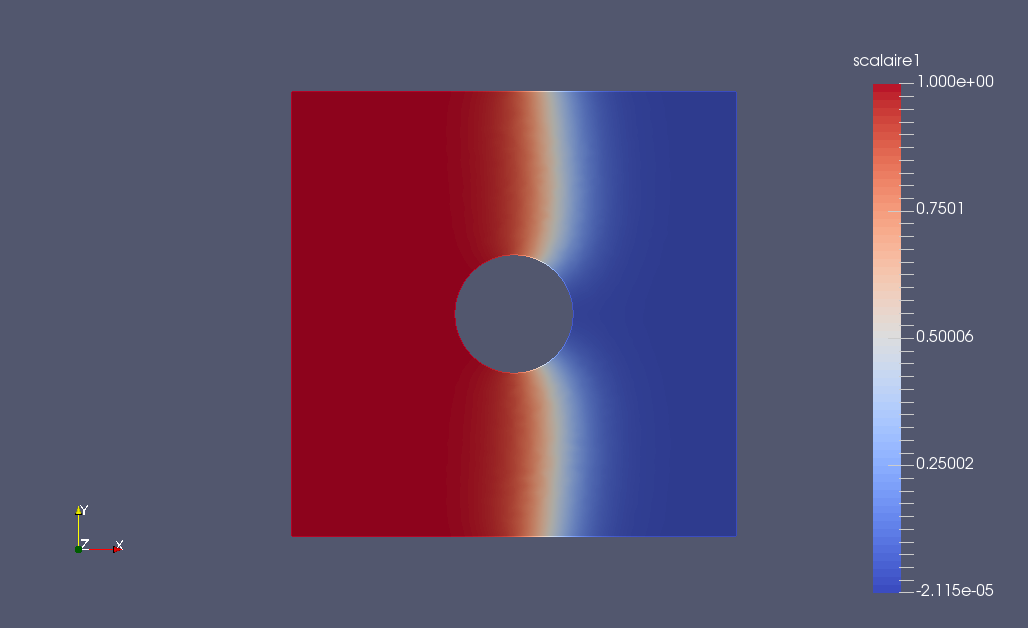}}
 \subfloat[$t=200$]{\includegraphics[width=4.2cm, height=2.5cm]{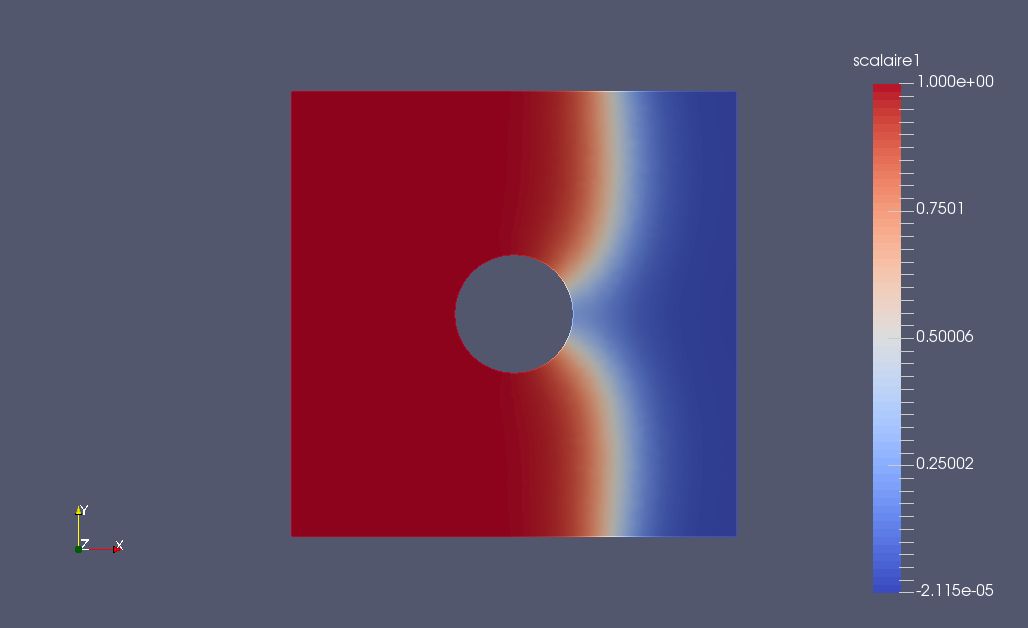}}
 \subfloat[$t=240$]{\includegraphics[width=4.2cm, height=2.5cm]{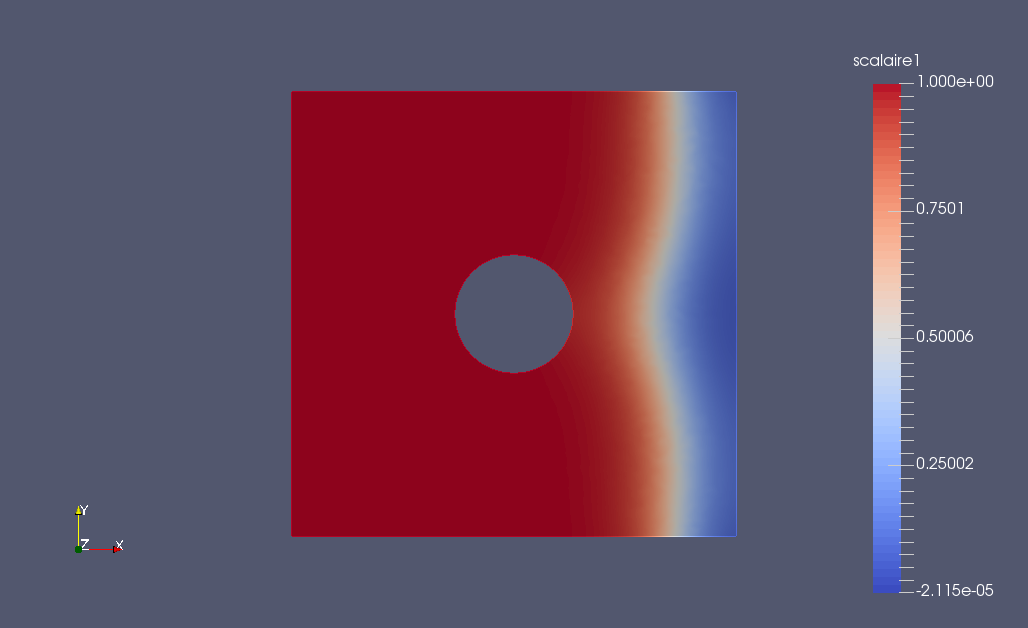}}
 \subfloat[$t=280$]{\includegraphics[width=4.2cm, height=2.5cm]{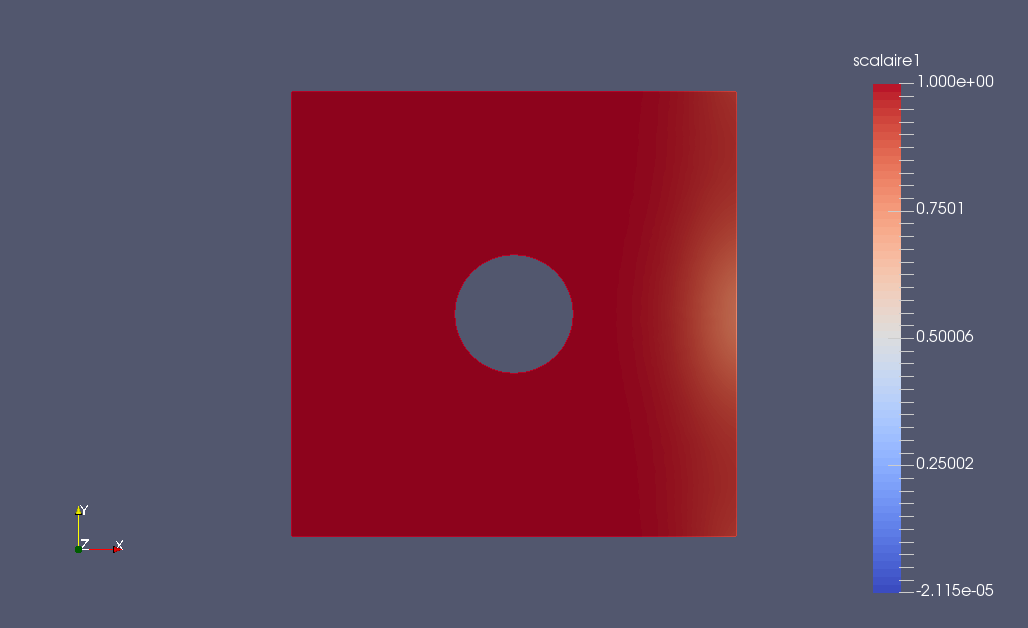}}
 \caption{Simulation of the singular non local evolution problem \eqref{NONLOCAL:EVOL:EQ} where the singular Levy kernel has been replaced  by a non singular measure $\frac{1}{\delta+|z|^{n+2s}}$ with $\delta=0.01$, the bistable non-lineartity $f$ is a cubic non-linearity $f(s)=s(s-0.1)(1-s)$ and the initial condition is of Heaviside type. We can see that the density $u(t,x)$ tends to 1 on the all space and  the influence of the obstacle on the shape of the transition}
 \label{fig-obs}
 \end{figure}

In this note, we deal with qualitative properties of the stationary solutions of equation~\eqref{NONLOCAL:EVOL:EQ}, together with some asymptotic limiting conditions at infinity similar to those appearing in~\eqref{EQ:BHM0}. Namely, we will be mainly concerned with  solutions of 
\begin{align}
\left\{
\begin{array}{rl}
\opl{u}+f(u)=0 & \text{in }\R^n\setminus K,\vspace{3pt}\\
0\leq u\leq 1 & \text{in }\R^n\setminus K, \vspace{3pt}\\
u(x)\to1 & \text{as }|x|\to+\infty.
\end{array} \label{EQ:BHM1}
\right.
\end{align}

with, when necessary,  the additional homogeneous Neumann boundary condition
\begin{equation}\label{EQ:BC}
\nabla u\cdot\nu=0 \qquad \text{ on }\qquad\partial K.
\end{equation}

\subsection{General assumptions}\label{SE:assumptions}
Let us now state precisely the  assumptions we will use. Along this note, we suppose that the  domain $K$ is a smooth (at least  $C^{0,1}$) bounded compact domain of $\R^n$, and that   $f$ is a smooth bistable non-linearity, that is $f$ will always satisfies

\be\label{C1}
f\in C^1([0,1]),\ \ f(0)=0=f(1)\ge 0,\ \ f'(1)<0,
\ee
\be\label{C2}
\left\{
\begin{array}{l}
\exists\,\theta\in(0,1),\ \ f(0)=f(\theta)=f(1)=0,\ \ f<0\hbox{ in }(0,\theta),\ \ f>0\hbox{ in }(\theta,1),\vspace{3pt}\\
\displaystyle\int_{0}^1f>0,\ \ f'(0)<0,\ \ f'(\theta)>0.\ \ \vspace{3pt}\\
\end{array}
\right.
\ee

Observe that the assumption on $f$ implies that the associated potential  is  unbalanced which is a necessary condition to observe the propagation of a front with a positive speed \cite{Achleitner2015,Gui2015a}. Thus it seems reasonable to assume such condition in our setting since we expect that the  solution $u$ of \eqref{EQ:BHM1} reflect the outcome of the invasion of the population  in the environment $\R^n\setminus K$.

\section{Main results}\label{MAIN RES}

For the local problem~\eqref{EQ:BHM0}, the  Liouville property obtained in \cite{Berestycki2009} says that $u=1$ in $\overline{\R^n\setminus K}$ under some geometric conditions on $K$,in particular when $K$ is convex.
Similar Liouville type property was recently obtained for continuous solution of \eqref{EQ:BHM1} when the singular kernel $\frac{1}{|z|^{n+2s}}$ is replaced by a non negative integrable  kernel $J$, i.e.  $J\in L^1(\R)$, see \cite{Brasseur2016}. 
More precisely, if  $J$ is assume to satisfy the assumptions below

\be\label{C4}\left\{\baa{l}
J\in L^1(\R^n)\hbox{ is a non-negative, radially symmetric kernel with unit mass},\vspace{3pt}\\
\hbox{there are }0\le r_1<r_2\hbox{ such that }J(x)>0\hbox{ for a.e. }x\hbox{ with }r_1<|x|<r_2,\eaa\right.
\ee
and there exists a function $\phi\in C(\R)$ satisfying
\be\label{C5}
\left\{
\begin{array}{l}
J_1\ast \phi-\phi+f(\phi)\geq 0\text{ in }\R, \vspace{3pt}\\
\phi\hbox{ is increasing in }\R,\ \ \phi(-\infty)=0,\ \ \phi(+\infty)=1,
\end{array}
\right.
\ee
where $J_1\in L^1(\R)$ is the non-negative even function with unit mass given for a.e. $x\in\R$ by
$$J_1(x):=\int_{\R^{n-1}} J(x,y_2,\cdots,y_n)\,\D y_2\cdots\D y_n.$$
 then in \cite{Brasseur2016} the authors prove the following

\begin{theorem}[Brasseur, Coville, Hamel,Valdinoci \cite{Brasseur2016}]\label{TH:JLIOUVILLE}
Let~~$K\subset\R^n$ be a compact convex set. Assume that  $f$ satisfies \eqref{C1} and \eqref{C2} and $J$ satisfies \eqref{C4} and \eqref{C5} and let
\begin{align}
u\in C(\overline{\R^n\setminus K},[0,1]) \label{continue}
\end{align}
be a function satisfying
\begin{equation}\label{nonlocJ}
\begin{cases}
&\ds{\int_{\R^n\setminus K}J(x-y)(u(y)-u(x))\,dy} +f(u(x))\leq 0  \quad \text{ for }\, x\in \overline{\R^n\setminus K}, \vspace{3pt}\\
&u(x)\to1  \quad \text{ as }\,|x|\to+\infty.
\end{cases}
\end{equation}
Then, $u=1$ in $\overline{\R^n\setminus K}$.
\end{theorem}

Observe that the problems \eqref{nonlocJ} and \eqref{EQ:BHM1} only differ in their formulation by the singularity of the kernels used. In particular, the problem \eqref{nonlocJ}  can be reformulated in to the framework  of problem \eqref{EQ:BHM1} since for all $J\in L^1(\R^n)$ and for all  $x \in \overline{\R^n\setminus K}$ and $u\in L^{\infty}(\R^n)$
$$\lim_{\eps\to }\int_{\R^n\setminus K,|x-y|>\eps}J(x-y)(u(y)-u(x))\,dy= \int_{\R^n\setminus K}J(x-y)(u(y)-u(x))\,dy.$$

Therefore, it is expected that~\eqref{nonlocJ} and~\eqref{EQ:BHM1} share some common properties. On of the goals of the present note is to extend the results known for ~\eqref{nonlocJ} to the solutions of \eqref{EQ:BHM1} and when possible to highlight the role of the singularity of the kernel in this context.

Our main results show that under the right regularity assumptions  we can transposed the results  of theorem \ref{TH:JLIOUVILLE} to  solutions to \eqref{EQ:BHM1}.  Namely, we first  prove that    

\begin{theorem}\label{TH:LIOUVILLE}
Let~~$K\subset\R^n$ be a compact smooth convex set ($C^{0,1}$). Assume~\eqref{C1},~\eqref{C2} and  $s \in (0,\frac{1}{2})$. Let $u \in C^{0,\beta}(\R^n\setminus K,[0,1])$ with $\beta>2s$ be a function satisfying
\begin{align}
\left\{
\begin{array}{rl}
\opl{u}+f(u)\leq 0  & \text{ in }\,\R^n\setminus K, \vspace{3pt}\\
u(x)\to 1  & \text{ as }\,|x|\to+\infty.
\end{array}
\right.  \label{LIM:u}
\end{align}

Then, $u=1$ in $\overline{\R^n\setminus K}$.
\end{theorem}

Our second result complete the picture, namely we show that

\begin{theorem}\label{TH2:LIOUVILLE}
Let~~$K\subset\R^n$ be a compact smooth convex set ($C^{0,1}$). Assume~\eqref{C1},~\eqref{C2} and  $s \in \left[\frac{1}{2},1\right)$. Let $u \in C^{1,\beta}(\R^n\setminus K,[0,1])$ with $\beta>2s-1$
be a function satisfying
\begin{align}
\left\{
\begin{array}{rl}
\opl{u}+f(u)\leq 0  & \text{ in }\,\R^n\setminus K, \vspace{3pt}\\
\nabla u\cdot\nu=0 & \text{ on }\,\partial K,\vspace{3pt}\\
u(x)\to 1  & \text{ as }\,|x|\to+\infty.
\end{array}
\right.  \label{LIM2:u}
\end{align}

Then, $u=1$ in $\overline{\R^n\setminus K}$.
\end{theorem}

We can  already see clearly  the effect of the singularity of the kernel.  Indeed, unlike the non local operators with integrable kernel  the $s$-fractional Laplacian is  well defined in $\overline{\R^n\setminus K}$  only for regular function, i.e. $u$ should be at least $C^{0,\beta}$. In this singular setting,  requiring that the super-solution $u$ is solely continuous is  not enough.  

These results complete our knowledge on the validity of such type of Liouville property for a broad class of reaction diffusion equation. 
They show some universality of such type of property   and prove  that such rigidity type result can be viewed as an intrinsic  property of the problem which can be related to a generic property of the equation rather than a special property of the diffusion process considered.
\subsection{Further comments and strategy of proofs}
Prior to proving these results, let us make some comments on our hypotheses and highlight some of the differences that arise when the singular measure $\frac{1}{|z|^{n+2s}}$ is replaced by an integrable kernel $J$.

First, let us observe that thanks to the regularising property of the regional fractional Laplacian $\l$, see \cite{Caffarelli2009,Guan2005,Guan2006} the continuity assumption made on $u$ can be easily weakened when $u$ is assumed to be a solution to \eqref{EQ:BHM1} instead of a super-solution. Thus, in this situation, the result of Theorems \ref{TH:LIOUVILLE} and \ref{TH2:LIOUVILLE} hold as well  for bounded  solution $u$ that satisfies the equation  \eqref{LIM:u} respectively \eqref{LIM2:u}  in the sense of viscosity solutions. 
Note that contrary to the regional fractional Laplacian $\l$, the nonlocal operator $\opm{u}:=\int_{\R^n\setminus K}J(x-y)(u(y)-u(x))\,dy$ has no regularising properties and as a consequence weakening the regularity assumption on the solution $u$ is a hard task which, for the moment, can only been   achieved  by imposing further restrictions on  the data $f$ and $J$. Nevertheless, in this non regularising context, the regularity of the obstacle $K$ is no more an issue and $K$ can be any arbitrary convex domain. The regularisation effect on the solutions induced by the singularity is in fact the only main distinction between the problem \eqref{nonlocJ} and the singular problems \eqref{LIM:u} and \eqref{LIM2:u}. 

This distinction appears also clearly in the set of assumption needed for the existence of monotone travelling front with a positive speed, which  is an essentiel key element of the proof of Theorem \ref{TH:JLIOUVILLE}. In particular, as already mentioned at Section~\ref{SE:assumptions}, assumptions \eqref{C1} and \eqref{C2} are actually  necessary and sufficient  for the existence of a  travelling wave solution with positite speed $c$ to the one dimentional fractional equation, i.e. a monotone solution to   $$c\,\partial_z\varphi=\partial_z^{s}\varphi+f(\varphi),$$ with a positive speed $c>0$. Such assumptions are not any more sufficient in the context of \eqref{nonlocJ}, where there exists data $f$ and $J$ that satisfy \eqref{C1}--\eqref{C2} for which only discontinuous  null speed fronts exist.  

Let us also note that, similar condition are also necessary and suffisent for the existence of one dimentional travelling wave with positive speed  for the local problem, namely solution of  $c\hspace{0.1em}\varphi'=\varphi''+f(\varphi)$ with positive speed (see e.g. \cite{Aronson1978}). This fact, then suggest a strong connexion between the regularity of the front and the minimal set of assumptions that are required to produced a front of positive speed.

Let us emphasize that the motivation behind these assumptions are that, by analogy with the local problem \eqref{EQ:BHM}, we expect a solution to \eqref{EQ:BHM1} to be the large time limit of an entire solution to the evolution problem \eqref{NONLOCAL:EVOL:EQ} which behaves like $\varphi(x_1+ct)$ when $t\to-\infty$. For this interpretation to even make sense it is necessary to work in a setting where the function $\varphi$ exists.

Let us now say a word on our  strategy of  proofs. The proofs are a rather straightforward adaptation of the arguments developed in \cite{Brasseur2016} for the non local obstacle problem \eqref{nonlocJ}.  The main idea is to compare by means of adequate  sliding method,  a family of planar function of the type $\varphi(x\cdot e-r)$, where $e\in\partial B_1$, $r\in\R$ and $\varphi\in C^{1,1}(\R)$ a given monotone function  with a given super-solution $u$.   
To adapt such technique to our situation,  we need first to verify that, as  proved for \eqref{nonlocJ} similar comparison principles in half-spaces hold true as well  for the fractional equations \eqref{LIM:u} and \eqref{LIM2:u}.  

The outline of this note will be as follow. In  Section~\ref{4} we provide  several comparison principles  and recall some known results on the 1d travelling fronts for fractional bistable equation.
Then we prove in Section ~\ref{8}, following the arguments developed in \cite{Brasseur2016}, we prove  the Liouville property described in  Theorems ~\ref{TH:LIOUVILLE} and \ref{TH2:LIOUVILLE}. 


\section{Some mathematical background}\label{4}

In this section, we start by collecting some comparison principles that fit for our purposes and to shortened the presentation we only fully state  the necessary comparison principle for regional fractional Laplacian with $s \in \left[\frac{1}{2},1\right)$. Throughout this section, $K$ is any compact subset of $\R^n$, $f$ is any $C^1(\R)$ function.  
 
We start with a weak maximum principle 

\begin{lemma}[Weak maximum principle]\label{WEAK}
Assume that $s \in \left[\frac{1}{2},1\right)$  and
\begin{equation}\label{well+2}
{\mbox{$f'\le -c_1$ in $[1-c_0,+\infty),\ $ for some $c_0>0$, $c_1>0$.}}
\end{equation}
Let~$H\subset\R^n$ be an open affine half-space such that $K\subset\subset H^c=\R^n\setminus H$. Let $u,v\in L^{\infty}(\R^n\setminus K)\cap C^{1,\beta}\big( \overline{H} \big)$ for some $\beta>1-2s$ be  such that

\begin{equation}\label{H:EQ}
\left\{\baa{ll}
\opl{u}+f(u)\leq 0 & {\mbox{ in }}\overline{H},\vspace{3pt}\\
\opl{v}+f(v)\ge 0 & {\mbox{ in }}\overline{H}.\eaa\right.
\end{equation}
Assume also that
\begin{equation}\label{up:c0}
u\ge 1-c_0 \quad{\mbox{ in }}\overline{H},
\end{equation}
that
\begin{equation}\label{H:lim}
\limsup_{|x|\to+\infty} \big(v(x)-u(x)\big)\le 0
\end{equation}
and that
\begin{equation}\label{H:SI}
v\le u\quad{\mbox{ a.e. in }} H^c\setminus K.
\end{equation}
Then,~$v\le u$ a.e. in~$\R^n\setminus K$.
\end{lemma}

The next lemma is concerned with a strong maximum principle.

\begin{lemma}[Strong maximum principle]\label{STRONG:LE23}
Assume that $s \in \left[\frac{1}{2},1\right)$  and let $H\subset\R^n$ be an open affine half-space such that $K\subset\subset H^c$. Let $u,v\in L^{\infty}(\R^n\setminus K)\cap C^{1,\beta}\big( \overline{H} \big)$ for some $\beta>1-2s$ be  such that \eqref{H:EQ} holds true. Assume also that
\begin{equation}\label{3beta:2}
v\le u \quad{\mbox{ a.e. in }} \R^n\setminus K
\end{equation}
and that there exists~$\bar x\in\overline{H}$ such that $v(\bar x)=u(\bar x)$. Then,
$$v= u \quad\mbox{ a.e. in } H.$$
\end{lemma}

These comparison principles are in essence identical to the one derived in \cite{Brasseur2016} and as such we  point the interested reader  to \cite{Brasseur2016} for a detailed proof of these results.

\begin{remark}
The above comparison principles have only been stated for regional fractional operators  with exponent $s\in \left[\frac{1}{2},1\right)$. Identical weak and strong maximum principles can be formulated for the regional fractional Laplacian with exponent $s\in \left(0,\frac{1}{2}\right)$ as soon as we impose the adequate regularity  to the functions $u$ and $v$ in order to properly define the regional fractional Laplacian of $u$ and $v$. In such case, the above statement will holds true if instead of having $u,v\in L^{\infty}(\R^n\setminus K)\cap C^{1,\beta}\big( \overline{H} \big)$ we  assume that $u,v\in L^{\infty}(\R^n\setminus K)\cap C^{0,\beta}\big( \overline{H} \big)$ with $\beta>2s$.
\end{remark}

Lastly, we recall some known result on the existence and properties of  travelling fronts  $\varphi(x\cdot e +ct)$, solution of the fractional evolution equation
\begin{equation*}
\partial_t u(t,x)=\Delta^s u(t,x) +f(u(t,x)) \quad \text{for} \quad t\in \R, x\in \R^n  
\end{equation*}
that is, solution of the following 
\begin{align} \label{fractw}
&-c\partial_z\varphi(z)+\partial^s_z \varphi(z) +f(\varphi(z)) =0 \quad \text{ for }\quad z\in \R\\
&\lim_{z\to +\infty}\varphi(z)=1, \qquad  \lim_{z\to -\infty}\varphi(z)=0 \label{fractw-lim}
\end{align}
where $\partial^s_z\varphi$ denotes the one dimentional $s-$fractional Laplacian. The existence, uniqueness and some asymptotic properties of such solution $\varphi$ have been obtained in several context \cite{Achleitner2015,Cabre2015,Gui2015a,Palatucci2013}. The next statement is a summary of these results.

\begin{theorem}[Fractional Travelling wave \cite{Achleitner2015,Cabre2015,Gui2015a,Palatucci2013}]\label{TW:TH} 
Assume $f$ is a bistable function that satisfies \eqref{C1}  and \eqref{C2} and let $s\in (0,1)$. Then there exists a unique $c\in \R$ and a monotone smooth (at least $C^{1,1}$) increasing function $\varphi$ such that $(c,\varphi)$ is a solution to \eqref{fractw}--\eqref{fractw-lim}.
Moreover, if $f$ is unbalanced with $\int_{0}^1f(s)\,ds>0$, then $c>0$. 
\end{theorem}

As a trivial consequence of the existence of a smooth front of positive speed, for any separating open affine half-space  $H \subset \R^n$ such that $K\subset\subset H^c$, we can derive a family of function which will be a "sub-solution to the problem \eqref{EQ:BHM1}" for all $x\in H$. More precisely, let $H$ be an affine subspace of $\R^n$ such that $K\subset H^c$. By definition of the affine space, there exists a unit vector $e\in \partial B_1$ and $x_0\in \R^n$ such that  $H=x_0+H_e$ with $H_e$ an open-halfspace of direction $e$, i.e. 
$$ H_e:=\{x\in \R^n| x\cdot e\ge 0\}.$$  
 
 For this direction $e$ and  for all real $r \in \R$,  we can define the family of  functions $\phi_{e,r}(x):=\varphi(x\cdot e -r)$ where $\varphi$ is the smooth increasing profile obtained in Theorem \ref{TW:TH}. By construction, since $\varphi$ is monotone increasing  we  have 
 \begin{equation}\label{ineq}
  \forall  \, x\in H, \forall\, y\in K\quad \phi_{r,e}(y)-\phi_{r,e}(x)\le 0.
   \end{equation} 
 In addition, we can check that for all $x\in H$, we have
 \begin{align*}
 \opl{\phi_{r,e}}(x)&= \lim_{\epsilon\to 0} \int_{\R^n\setminus K, |x-y|>\epsilon}\frac{\phi_{r,e}(y)-\phi_{r,e}(x)}{|x-y|^{n+2s}}\,dy  \\
 &= \lim_{\epsilon\to 0} \int_{\R^n, |x-y|>\epsilon}\frac{\phi_{r,e}(y)-\phi_{r,e}(x)}{|x-y|^{n+2s}}\,dy - \int_{K, |x-y|>\epsilon}\frac{\phi_{r,e}(y)-\phi_{r,e}(x)}{|x-y|^{n+2s}}\,dy\\
 &= \Delta^s\phi_{r,e}(x) - \int_{K}\frac{\phi_{r,e}(y)-\phi_{r,e}(x)}{|x-y|^{n+2s}}\,dy.
\end{align*}
which combined with \eqref{ineq} enforces  
$$\opl{\phi_{r,e}}(x) \ge \Delta^s\phi_{r,e}(x) \quad \text{for} \quad x\in H.$$

Hence, for all $x\in H$, we get 

\begin{equation}\label{eq:subsol}
\opl{\phi_{e,r}}(x) +f(\phi_{e,r})\ge \Delta^s\phi_{e,r}(x)+f(\phi_{e,r})=\partial^s\varphi(x.e-r)+f(\varphi(x.e-r))=c\partial_z\varphi(x\cdot e-r)>0.
 \end{equation}

\section{The case of convex obstacles: proofs of the main Theorem}\label{8}

In this section, based on the arguments introduced in  \cite{Brasseur2016} we sketch the proof our main results (Theorems \ref{TH:LIOUVILLE} and \ref{TH2:LIOUVILLE}). 
The proofs of the Theorems \ref{TH2:LIOUVILLE} and \ref{TH:LIOUVILLE} being identical, we only sketch the proof of Theorem \ref{TH2:LIOUVILLE}.

But before we start our argumentation,let us first start with the following simple observation.

\begin{lemma}\label{CLAIM0}
Let $s\in\left[\frac{1}{2},1\right)$, $K\subset\R^n$ be a smooth compact convex set and assume~\eqref{C1} and~\eqref{C2}. Let  $u\in C^{1,\beta}(\R^n\setminus K),[0,1])$ with $\beta>1-2s$ be such that  
\begin{align}
\opl{u}+f(u)\le0\,  & \,\text{ in }\,\R^n\setminus K, \label{EQ:Liouville} \vspace{3pt}\\
\nabla u\cdot \nu =0\,  & \,\text{ in }\,\partial K, \label{EQ:Liouvilleboundary} \vspace{3pt}\\
u(x)\to1\,  & \,\text{ as }\,|x|\to+\infty.  \label{EQ:Liouvillelim}
\end{align}
Then  there exists $\gamma\in(0,1]$ such that $\gamma\leq u\leq 1$ in $\overline{\R^n\setminus K}$.
\end{lemma}

The proof of this Lemma being an elementary adaptation of the argument used in \cite{Brasseur2016}, we will refer to \cite{Brasseur2016}  for its proof.

We now turn to the proof of Theorem ~\ref{TH2:LIOUVILLE} .

\begin{proof}[Proof of Theorem~$\ref{TH2:LIOUVILLE}$]

Let us fix $s\in \left[\frac{1}{2},1\right)$ and let $K$, $f$,  and $u$ be as in Theorem~\ref{TH2:LIOUVILLE}. 
Let us now follow the argument developed in \cite{Brasseur2016}.
Firstly, without loss of generality, one can assume by~\eqref{C1} that $f$ is
extended to a $C^1(\R)$ function satisfying~\eqref{well+2}. Secondly, by~\eqref{LIM:u} and the boundedness of $K$, there exists~$R_0>0$ large enough so that~$K\subset B_{R_0}$ and~$u\ge 1-c_0$ in~$\R^n\setminus B_{R_0}$, where $c_0>0$ is given in~\eqref{well+2}.\par

We now proceed  by contradiction, and suppose that
\begin{align}
\inf_{\overline{\R^n\setminus K}}\,u<1. \label{hypoko}
\end{align}
{F}rom~\eqref{LIM2:u} and~\eqref{hypoko}, together with the continuity of $u$, there exists then $x_0\in\overline{\R^n\setminus K}$ such that
$$u(x_0)=\min_{\overline{\R^n\setminus K}} u\in[0,1).$$
We observe that, by Lemma~\ref{CLAIM0}, one has $u(x_0)>0$. 
In addition,  since $K$ is convex, there exists $e\in \partial B_1$ such that  $K\subset H_{e}^c$, where $H_{e}$ is the open affine half-space defined by
$$H_{e}:=x_0+\big\{x\in\R^n;\ x\cdot e> 0\big\}.$$
As in  section \ref{4}, let us define for all $r\in \R$, the 
family of functions 
$$ \phi_r(x):=\phi_{r,e}(x)=\varphi(x\cdot e-r),\ \ x\in\R^n,$$
  where $\varphi$ is  a smooth monotone increasing function given by Theorem \ref{TW:TH}.
Note that by construction, since  $K\subset H_{e}^c$,  we can check (as in the section \ref{4}) that for any $r\in \R$, $\phi_r$ satisfies

\begin{equation}
\label{eq:subsolr}
\opl{\phi_r}(x)+f(\phi_r(x))>0 \quad \text{ for }\quad x \in \overline{H_e}.
\end{equation}

First, we  claim  that 
\begin{claim}\label{eq_r0}
There exists $r_0\in \R$ such that $\phi_{r_0}\le u$ in $\overline{\R^n\setminus K}$.
\end{claim}

Again the proof of this claim is an elementary adaptation of a  proof done in \cite{Brasseur2016} that  for the sake of clarity we give the details.  
\begin{proof}
First let us define  $H:=x_1+ H_e$ with $x_1$ to be chosen such that $B_{R_0} \subset H^c$. Let us fix $x_1$ such that $H\subset\subset H_e$. 
By construction the function $\varphi$ is monotone increasing and satisfies $\lim_{z\to -\infty}\varphi(z)=0$. So we can find  $r_0>>1$ such that 
$\phi_{r_0}(x)=\varphi(x\cdot e -r_0)\le u(x_0)\le u(x)$ for all $x \in H^c$. 
Now thanks to our choice of $x_1$ we have $H\subset\subset H_e$ and from \eqref{eq:subsolr} we deduce 
$$\begin{cases}
\opl{u}(x)+f(u(x))\le 0  & \text{for} \quad x \in \overline{H},\vspace{3pt}\\
\opl{\phi_{r_0}}(x)+f(\phi_{r_0}(x))  > 0 & \text{for} \quad x \in \overline{H}, \vspace{3pt}\\
u(x)\ge \phi_{{r_0}}(x) & \text{for} \quad x \in H^c\setminus K,
\end{cases}
$$
We then get the desired results by applying the weak-maximum principle (Lemma \ref{WEAK}).
\end{proof}

Equipped with the claim \ref{eq_r0}, we can now define  the following quantity
$$ r^* := \inf\big\{ r\in \R\,;\ \phi_r\le u\mbox{ in }\overline{\R^n\setminus K}\big\}.$$

All the game now is to show that $r^*=-\infty$. 
So, we claim that
\begin{claim}\label{GOAL}
$
r^*=-\infty.
$
\end{claim}

Assume for the claim is true, then the proof of Theorem \ref{TH2:LIOUVILLE} is thereby complete. Indeed,  from  this claim  we infer that~$\phi_r\le u$ in $\overline{\R^n\setminus K}$ for any~$r\in\R$. In particular, recalling that $\varphi(+\infty)=1$, we get that
$$1> u(x_0)\ge \lim_{r\to-\infty} \phi_r(x_0) =\lim_{r\to-\infty} \varphi(x_0\cdot e-r) =1,$$
a contradiction. Therefore,~\eqref{hypoko} can not hold. In other words, $\inf_{\overline{\R^n\setminus K}}u=1$, i.e. $u=1$ in $\overline{\R^n\setminus K}$ proving thereby Theorem~\ref{TH2:LIOUVILLE}.
\end{proof}

Let us now conclude our proof by establishing the Claim \ref{GOAL}. Again, the proof of this last claim is done by a  very elementary adaption of the arguments used to prove Theorem \ref{TH:JLIOUVILLE}. As a consequence we will only highlights the main differences.

\begin{proof}[Proof of the Claim \ref{GOAL}]
The proof of~\eqref{GOAL} is by contradiction. We assume that~$r^*\in\R$. Then, there exists a sequence~$(\varepsilon_j)_{j\in\N}$ of
positive real numbers such that~$\phi_{r^*+\varepsilon_j}(x)=\varphi(x\cdot e-r^*-\varepsilon_j)\le u(x)$ for all~$x\in\overline{\R^n\setminus K}$ and $\varepsilon_j\to0$ as $j\to+\infty$. Thus, passing to the limit as~$j\to+\infty$, we obtain that
$$\phi_{r^*}(x)\le u(x)\quad{\mbox{ for all }} x\in \overline{\R^n\setminus K}.$$
Let us denote $H$ the open affine half-space
$$H=\big\{x\in \R^n;\ x\cdot e>R_0\big\}.$$
Notice that $\overline{H}\cap K=\emptyset$ and that $u$ is well defined and continuous in $\overline{H}$.
We also observe that, by construction,
\begin{align}
\sup_{H^c}\phi_{r^*}<1. \label{bornsup}
\end{align}

Two cases may occur.\par
{\it Case 1:} $\inf_{H^c\setminus K}(u-\phi_{r^*})>0$. In this situation, the argument is identical as in for \eqref{nonlocJ}, and we point the reader to \cite{Brasseur2016} for the details.

{\it Case 2:} $\inf_{H^c\setminus K}(u-\phi_{r^*})=0$. In this situation, by~\eqref{EQ:Liouvillelim} and~\eqref{bornsup}, and by continuity of $u$ and $\phi_{r^*}$,
there exists a point $\bar x\in\overline{H^c\setminus K}$ such that $u(\bar x)=\phi_{r^*}(\bar x)$. Note that $\bar x \in \overline{H_e}$, since otherwise $\bar x\in \R^n\setminus\overline{H_e}$, namely $\bar{x}\cdot e<x_0\cdot e$, and
the chain of inequalities
$$u(\bar x)=\phi_{r^*}(\bar x)<\phi_{r^*}(x_0)\le u(x_0)=
\min_{\overline{\R^n\setminus K}}u$$
leads to a contradiction. Therefore, we have $\phi_{r^*}\le u$ in $\overline{\R^n\setminus K}$ with equality at a point $\bar{x}\in\overline{\R^n\setminus K}\cap\overline{H_e}$. Again, two situations can occur either  $\bar x \in \R^n\setminus K$ or $\bar x\in \partial K$. 
Assume for the moment that the latter situation occurs. Then thanks to convexity of $K$, $\nu(\bar x)$ the outward normal to $\partial K$  at $\bar x$ is then the vector $e$, i.e. $\nu(\bar x)=e$ and thanks to \eqref{LIM2:u} $\nabla u\cdot \nu(\bar x)=0$  we deduce 
$$0\le \nabla (u-\phi_{r^*})\cdot \nu(\bar x) \le -\nabla (\phi_{r^*})\cdot \nu(\bar x)=-\varphi'(\bar x\cdot e-r^*) <0. $$
This contradiction then rules out this situation. Lastly assume that $\bar x \in \R^n\setminus K$, then in this situation
  $K\subset\subset H_e^c$ and $\varphi_{r^*}$ and $u$ satisfy respectively
\begin{equation*}
\left\{\baa{rl}
\opl{u}+f(u)\leq 0 & {\mbox{ in }}\overline{H_e},\vspace{3pt}\\
\opl{\phi_{r^*}}+f(\phi_{r^*})> 0 & {\mbox{ in }}\overline{H_e}\ \ \hbox{(by }\eqref{eq:subsolr}\hbox{)},\eaa\right..
\end{equation*}
In particular, it  follows  from the strong maximum principle (Lemma~\ref{STRONG:LE23}) that~$\phi_{r^*}=u$ in~$\overline{H_e}$. Thus, for any $e^{\perp}\in\partial B_1$ such that $e^\perp\cdot e=0$, one infers from~\eqref{EQ:Liouvillelim} and the definition of $\phi_{r^*}$ that
$$ 1=\lim_{t\to+\infty} u(x_0+t\,e^{\perp})=\lim_{t\to+\infty} \phi_{r^*}(x_0+t\,e^{\perp}) =\phi_{r^*}(x_0)<1.$$
This last contradiction then rules out also this situation and therefore  rules out Case~2 too.\par
\end{proof}

\begin{remark}
The proof of Theorem \ref{TH:LIOUVILLE} is identical to the one given in \cite{Brasseur2016}. This is due to the fact that the $s-$fractional operator is well defined  and continuous up to the boundary of $\partial K$ when $s\in \left(0,\frac{1}{2}\right)$ and as such the strong maximum principle (Lemma \ref{STRONG:LE23}) holds also true for any half space $H$ such that $K\subset H^c$. In this situation, the case $\bar x \in \partial K$ does not need to be analysed separately from the other cases. 
\end{remark}

\ackno{ The author would like to thank Professor Changfeng Gui for raising to my attention this topic  during the MATRIX event ``Recent Trends on Nonlinear PDEs of Elliptic and Parabolic Type'' }

\bibliographystyle{amsplain}
\bibliography{proceeding.bib}

\end{document}